\documentclass[12pt]{amsart}
\usepackage{amscd,amssymb,amsmath,multicol,scalefnt,etex,tikz}
\newtheorem{thm}[equation]{Theorem}
\numberwithin{equation}{section}
\newtheorem{cor}[equation]{Corollary}

\newtheorem{prop}[equation]{Proposition}

\newtheorem{fig}[equation]{Figure}

\begin{document}
\setcounter{MaxMatrixCols}{26}
\raggedbottom \voffset=-.7truein \hoffset=0truein \vsize=8truein
\hsize=6truein \textheight=8truein \textwidth=6truein
\baselineskip=18truept

\def\mapright#1{\ \smash{\mathop{\longrightarrow}\limits^{#1}}\ }
\def\mapleft#1{\smash{\mathop{\longleftarrow}\limits^{#1}}}
\def\mapup#1{\Big\uparrow\rlap{$\vcenter {\hbox {$#1$}}$}}
\def\mapdown#1{\Big\downarrow\rlap{$\vcenter {\hbox {$\ssize{#1}$}}$}}
\def\mapne#1{\nearrow\rlap{$\vcenter {\hbox {$#1$}}$}}
\def\mapse#1{\searrow\rlap{$\vcenter {\hbox {$\ssize{#1}$}}$}}
\def\mapr#1{\smash{\mathop{\rightarrow}\limits^{#1}}}
\def\ss{\smallskip}
\def\vp{v_1^{-1}\pi}
\def\at{{\widetilde\alpha}}
\def\sm{\wedge}
\def\la{\langle}
\def\ra{\rangle}
\def\on{\operatorname}
\def\ol#1{\overline{#1}{}}
\def\spin{\on{Spin}}
\def\lbar{\ell}
\def\qed{\quad\rule{8pt}{8pt}\bigskip}
\def\ssize{\scriptstyle}
\def\a{\alpha}
\def\bz{{\Bbb Z}}
\def\im{\on{im}}
\def\ct{\widetilde{C}}
\def\ext{\on{Ext}}
\def\sq{\on{Sq}}
\def\eps{\epsilon}
\def\ar#1{\stackrel {#1}{\rightarrow}}
\def\br{{\bold R}}
\def\bC{{\bold C}}
\def\bA{{\bold A}}
\def\bB{{\bold B}}
\def\bd{{\bold d}}
\def\bh{{\bold H}}
\def\bQ{{\bold Q}}
\def\bP{{\bold P}}
\def\bx{{\bold x}}
\def\bo{{\bold{bo}}}
\def\si{\sigma}
\def\Vbar{{\overline V}}
\def\dbar{{\overline d}}
\def\wbar{{\overline w}}
\def\Sum{\sum}
\def\tz{tikzpicture}
\def\tfrac{\textstyle\frac}
\def\tb{\textstyle\binom}
\def\Si{\Sigma}
\def\s{\sigma}
\def\st{\widetilde{\sigma}}
\def\w{\wedge}
\def\equ{\begin{equation}}
\def\b{\beta}
\def\G{\Gamma}
\def\g{\gamma}
\def\k{\kappa}
\def\psit{\widetilde{\Psi}}
\def\tht{\widetilde{\Theta}}
\def\psiu{{\underline{\Psi}}}
\def\thu{{\underline{\Theta}}}
\def\aee{A_{\text{ee}}}
\def\aeo{A_{\text{eo}}}
\def\aoo{A_{\text{oo}}}
\def\aoe{A_{\text{oe}}}
\def\vbar{{\overline v}}
\def\endeq{\end{equation}}
\def\sn{S^{2n+1}}
\def\zp{\bold Z_p}
\def\A{{\cal A}}
\def\P{{\mathcal P}}
\def\cQ{{\mathcal Q}}
\def\cj{{\cal J}}
\def\zt{{\bold Z}_2}
\def\bs{{\bold s}}
\def\by{{\bold y}}
\def\bx{{\bold x}}
\def\be{{\bold e}}
\def\Hom{\on{Hom}}
\def\ker{\on{ker}}
\def\coker{\on{coker}}
\def\da{\downarrow}
\def\colim{\operatornamewithlimits{colim}}
\def\zphat{\bz_2^\wedge}
\def\io{\iota}
\def\Om{\Omega}
\def\Prod{\prod}
\def\e{{\cal E}}
\def\exp{\on{exp}}
\def\abar{{\overline a}}
\def\xbar{{\overline x}}
\def\ybar{{\overline y}}
\def\zbar{{\overline z}}
\def\Rbar{{\overline R}}
\def\nbar{{\overline n}}
\def\cbar{{\overline c}}
\def\qbar{{\overline q}}
\def\bbar{{\overline b}}
\def\et{{\widetilde E}}
\def\ni{\noindent}
\def\coef{\on{coef}}
\def\den{\on{den}}
\def\lcm{\on{l.c.m.}}
\def\vi{v_1^{-1}}
\def\ot{\otimes}
\def\psibar{{\overline\psi}}
\def\mhat{{\hat m}}
\def\exc{\on{exc}}
\def\ms{\medskip}
\def\ehat{{\hat e}}
\def\etao{{\eta_{\text{od}}}}
\def\etae{{\eta_{\text{ev}}}}
\def\dirlim{\operatornamewithlimits{dirlim}}
\def\gt{\widetilde{L}}
\def\lt{\widetilde{\lambda}}
\def\st{\widetilde{\sigma}}
\def\ft{\widetilde{f}}
\def\sgd{\on{sgd}}
\def\lfl{\lfloor}
\def\rfl{\rfloor}
\def\ord{\on{ord}}
\def\gd{{\on{gd}}}
\def\rk{{{\on{rk}}_2}}
\def\nbar{{\overline{n}}}
\def\lg{{\on{lg}}}
\def\cB{\mathcal{B}}
\def\cS{\mathcal{S}}
\def\cP{\mathcal{P}}
\def\N{{\Bbb N}}
\def\Z{{\Bbb Z}}
\def\Q{{\Bbb Q}}
\def\R{{\Bbb R}}
\def\C{{\Bbb C}}
\def\l{\left}
\def\r{\right}
\def\mo{\on{mod}}
\def\xt{\times}
\def\notimm{\not\subseteq}
\def\Remark{\noindent{\it  Remark}}

\def\*#1{\mathbf{#1}}
\def\0{$\*0$}
\def\1{$\*1$}
\def\22{$(\*2,\*2)$}
\def\33{$(\*3,\*3)$}
\def\ss{\smallskip}
\def\ssum{\sum\limits}
\def\dsum{\displaystyle\sum}
\def\la{\langle}
\def\ra{\rangle}
\def\on{\operatorname}
\def\o{\on{o}}
\def\U{\on{U}}
\def\lg{\on{lg}}
\def\a{\alpha}
\def\bz{{\Bbb Z}}
\def\eps{\varepsilon}
\def\bc{{\bold C}}
\def\bN{{\bold N}}
\def\nut{\widetilde{\nu}}
\def\tfrac{\textstyle\frac}
\def\b{\beta}
\def\G{\Gamma}
\def\g{\gamma}
\def\zt{{\Bbb Z}_2}
\def\zth{{\bold Z}_2^\wedge}
\def\bs{{\bold s}}
\def\bx{{\bold x}}
\def\bof{{\bold f}}
\def\bq{{\bold Q}}
\def\be{{\bold e}}
\def\lline{\rule{.6in}{.6pt}}
\def\xb{{\overline x}}
\def\xbar{{\overline x}}
\def\ybar{{\overline y}}
\def\zbar{{\overline z}}
\def\ebar{{\overline \be}}
\def\nbar{{\overline n}}
\def\rbar{{\overline r}}
\def\Mbar{{\overline M}}
\def\et{{\widetilde e}}
\def\ni{\noindent}
\def\ms{\medskip}
\def\ehat{{\hat e}}
\def\what{{\widehat w}}
\def\Yhat{{\widehat Y}}
\def\nbar{{\overline{n}}}
\def\minp{\min\nolimits'}
\def\mul{\on{mul}}
\def\N{{\Bbb N}}
\def\Z{{\Bbb Z}}
\def\Q{{\Bbb Q}}
\def\R{{\Bbb R}}
\def\C{{\Bbb C}}
\def\notint{\cancel\cap}
\def\cS{\mathcal S}
\def\cR{\mathcal R}
\def\el{\ell}
\def\TC{\on{TC}}
\def\dstyle{\displaystyle}
\def\ds{\dstyle}
\def\Wbar{\wbar}
\def\zcl{\on{zcl}}
\def\Vb#1{{\overline{V_{#1}}}}

\def\Remark{\noindent{\it  Remark}}
\title
{The symmetrized topological complexity of the circle}
\author{Donald M. Davis}
\address{Department of Mathematics, Lehigh University\\Bethlehem, PA 18015, USA}
\email{dmd1@lehigh.edu}
\date{March 15, 2017}

\keywords{Topological complexity}
\thanks {2000 {\it Mathematics Subject Classification}: 55M30, 55N25, 57M20.}

\maketitle
\begin{abstract} We determine the symmetrized topological complexity of the circle, using primarily just general topology.
 \end{abstract}
\section{Introduction}
Let $PX$ denote the space of all paths in a topological space $X$, and define $p:PX\to X\times X$ by $p(\s)=(\s(0),\s(1))$. If $V\subset X\times X$, a section $s:V\to PX$ is called a motion planning rule on $V$. The reduced topological complexity of $X$, $\TC(X)$, is 1 less than the minimal number of open sets $V$ covering $X\times X$ which admit motion planning rules. The notion of topological complexity was introduced by Farber in \cite{F} in unreduced form, but most recent papers have preferred the reduced notation. Topological complexity can be applied to robotics when $X$ is the space of configurations of a robot.

A set $V\subset X\times X$ is {\it symmetric} if $(x,y)\in V$ iff $(y,x)\in V$. A symmetric motion planning rule on such a set $V$ is one which satisfies $s(x_1,x_0)=\overline{s(x_0,x_1)}$. Here $\overline{\sigma}(t)=\sigma(1-t)$.

In \cite{BGRT}, (reduced) symmetrized topological complexity $\TC^\Sigma(X)$ of  $X$ was defined to be 1 less than the minimal number of symmetric open sets covering $X\times X$ which admit symmetric motion planning rules. We will prove the following new result.
\begin{thm}\label{main} $\TC^\Sigma(S^1)=2$.\end{thm}

An earlier variant, called symmetric topological complexity, $\TC^S(X)$, was introduced in \cite{FG}. Employing here the reduced TC terminology, $\TC^S(X)$ equals the minimal number of symmetric open sets covering $X\times X-\Delta$ admitting symmetric motion planning rules. Here $\Delta=\{(x,x)\in X\times X\}$ is the diagonal. This notion assumes that one motion planning rule chooses the constant path from $x$ to $x$, possibly extended over a small neighborhood of $\Delta$, and then considers separately symmetric paths between distinct points. The reduced version employed here has the $-1$ in the reduction which cancels the $+1$ from the section over the diagonal. As noted in \cite[Prop 4.2]{BGRT}, it is immediate that for all $X$
$$\TC^S(X)-1\le \TC^{\Si}(X)\le\TC^S(X).$$

The advantage of the $\TC^S(-)$ concept is that, with $P'X$ denoting the space of paths between distinct points of $X$, there is a
$\zt$-equivariant fibration $P'X\to X\times X-\Delta$ of free $\zt$-spaces. This leads to nice cohomological bounds for $\TC^S(-)$. In an email to the author, Michael Farber confirmed that he felt that the $\TC^{\Sigma}$ definition was ``more natural'' than $\TC^S$. One nice feature of $\TC^\Si$ is that it is a homotopy invariant (\cite[Prop 4.7]{BGRT}), whereas $\TC^S$ is not. In the paper \cite{G}, written at the same time as ours, Mark Grant
discusses more fully the relationships between $\TC^\Sigma$ and $\TC^S$. In that paper he develops cohomological bounds for $\TC^\Sigma$.

In \cite{FG}, it was shown that $\TC^S(S^n)=2$ for all $n\ge1$. Since cohomology shows that when $n$ is even, three (not necessarily symmetric) motion planning rules are required for $S^n$, we obtain that $\TC^\Si(S^n)=2$ when $n$ is even. In \cite[Expl 4.5]{BGRT} and in \cite[Expl 17.5]{R}, it was noted that for odd $n$, it was not known whether $\TC^\Si(S^n)=1$ or 2, and the case $n=1$ was given special attention as an ``Open Problem'' in \cite[17.6]{R}. Our contribution here is to resolve this open problem.
In \cite{G}, Grant has proved $\TC^\Si(S^n)=2$ for all $n$, including $n=1$, which required special methods.
We thank him for interesting and helpful discussion about his approach and ours.

\section{Our approach and an example}\label{sec2}
Our approach is to associate to a motion planning rule on an open subset of $S^1\times S^1$ a locally constant function $d$ on an open subset of $I\times I$ with certain properties, and then show (in the next section) that the domains of two such functions cannot cover $I\times I$.

Let $\rho:I\times I\to S^1\times S^1$ be the usual quotient map defined by $\rho(t,t')=(e^{2\pi it},e^{2\pi it'})$, and $e:\R\to S^1$ the usual covering map defined by $e(t)=e^{2\pi it}$.

\begin{prop}\label{dprop}
If $V\subset S^1\times S^1$ is a symmetric open set, and $s:V\to PS^1$ is a symmetric motion planning rule, there is a continuous function
$d:\rho^{-1}(V)\to\Z$ satisfying, for all points in its domain,
\begin{eqnarray}\label{C1}&&d(t,1)-d(t,0)=-1,\\
\label{C2}&&d(1,t)-d(0,t)=1,\\ \label{symm}\text{and}&&d(t',t)=-d(t,t').\end{eqnarray}
\end{prop}
\begin{proof} Suppose $\rho(t,t')\in V$ with $\s=s(\rho(t,t'))\in PS^1$. Let $\st:I\to\R$ satisfy $e\circ\st=\s$. Note that $\st(1)-\st(0)$ is independent of the choice of $\st$. Let
$$d(t,t')=\st(1)-\st(0)-(t'-t)\in\R.$$
Then
\begin{eqnarray*}e(d(t,t'))&=&\s(1)\s(0)^{-1}e^{2\pi i(t-t')}\\
&=&e^{2\pi it'}e^{-2\pi it}e^{2\pi i(t-t')}=1.\end{eqnarray*}
Therefore $d(t,t')\in \Z$.

To see continuity of $d$, first note that $\s$ varies continuously with $(t,t')$. Thus $\st(0)$ can be chosen to vary continuously with $(t,t')$, and hence so does $\st(1)$, by the Homotopy Lifting Theorem.

 Since $\rho(t,1)=\rho(t,0)$, the $\s$'s associated to these points are the same, and hence so are the two values of $\st(1)-\st(0)$. Now property (\ref{C1}) follows immediately from the change in $t'$, and (\ref{C2}) follows similarly. Property (\ref{symm}) is clear, since both $t'-t$ and $\st(1)-\st(0)$ are negated when $t$ and $t'$ are interchanged.\end{proof}

Since $d$ is a continuous integer-valued function, it is constant on connected sets, a fact which we will use frequently. Note that, by (\ref{C1}) and (\ref{C2}), $(t,1)$ is in the domain of $d$ iff $(t,0)$ is, and similarly for $(1,t)$ and $(0,t)$.

Next we provide an example of the functions $d$ associated to three motion planning rules whose domains cover the circle.
The rules for moving from $z$ to $z'$ are as follows.
\begin{itemize}
\item If $z$ and $z'$ are not antipodal, follow the geodesic.
\item If $z$ and $z'$ are not at the same horizontal level, let $w=\frac{z-z'}{|z-z'|}$ and $w'=-w$, and follow the geodesic from $z$ to $w$, then the path from $w$ to $w'$ which passes through $1$, then the geodesic from $w'$ to $z'$.
\item If $z$ and $z'$ are not at the same vertical level, let $w=\frac{z-z'}{|z-z'|}$ and $w'=-w$, and follow the geodesic from $z$ to $w$, then the path from $w$ to $w'$ which passes through $i=e^{i\pi/2}$, then the geodesic from $w'$ to $z'$.
\end{itemize}

The functions $d$ for these are as pictured in Figure \ref{fig1}:

\smallskip

\begin{fig}\label{fig1}

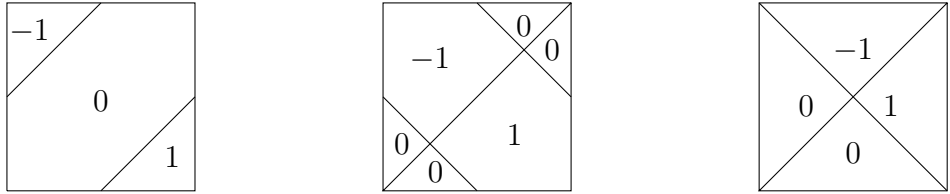
{\bf Domains of three motion planning rules.}

\begin{center}

\begin{\tz}[scale=2.5]
\draw (0,0) --(1,0) --(1,1) -- (0,1) -- (0,0);
\node at (.5,.48) {$0$}; \node at (.88,.18) {$1$}; \node at (.12,.85) {$-1$};
\draw (0,.5) -- (.5,1); \draw (.5,0) --(1,.5);
\draw (3,1) --(2,0) --(3,0) --(3,1) -- (2,1) -- (2,0);
\node at (2.7,.3) {$1$}; \node at (2.25,.7) {$-1$};
\node at (2.28,.1) {$0$}; \node at (2.10,.25) {$0$};
\node at (2.75,.88) {$0$}; \node at (2.9,.75) {$0$};
\draw (2.5,0) -- (2,.5); \draw (3,.5) --(2.5,1);
\draw (5,1) --(4,0) --(5,0) --(5,1) -- (4,1) -- (4,0);
\draw (5,0) --(4,1);
\node at (4.5,.2) {$0$}; \node at (4.5,.75) {$-1$};
\node at (4.25,.45) {$0$}; \node at (4.7,.45) {$1$};
\end{\tz}

\end{center}

\end{fig}

\noindent Points in $\partial I^2$ are in the domains except for $(0,0)$, $(1,0)$, $(0,1)$, and $(1,1)$ in the second and third, and $(0,\frac12)$, $(1,\frac12)$, $(\frac12,0)$, and $(\frac12,1)$ in the first and second.

For example, the region labeled ``1'' in the second square corresponds to points $(z,z')=(e^{2\pi it},e^{2\pi it'})$ with $t>t'$ and $\frac12<t+t'<\frac32$. One such point has $(t,t')=(\frac12,\frac18)$. For the second motion planning rule above, $w=e^{i9\pi/8}$, $w'=e^{i\pi/8}$, and $\si$ is a counterclockwise rotation from $z$ to $z'$, passing through $w$ and $w'$. Thus $\st(1)-\st(0)=\frac58$, and
$$d(\tfrac12,\tfrac18)=\tfrac58-(\tfrac18-\tfrac12)=1.$$
This is illustrated in Figure \ref{fig3}.

\begin{fig}\label{fig3}

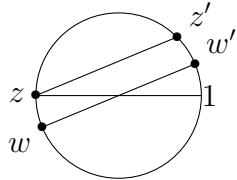
{\bf Circle illustrating a motion planning rule.}

\begin{center}

\begin{\tz}[scale=.55]
\draw (0,0) circle [radius=2];
\draw (1.43,1.424) --(-2,0) --(2,0);
\draw (-1.848,-.7654) --(1.848,.7654);
\node [left]  at (-2,0) {$z$};
\node [above right] at (1.43,1.424) {$z'$};
\node at (-2,0) {$\ssize\bullet$};
\node at (1.414,1.414) {$\ssize\bullet$};
\node at (1.848,.7654) {$\ssize\bullet$};
\node at (-1.848,-.7654) {$\ssize\bullet$};
\node [below left] at (-1.848,-.7654) {$w$};
\node [above right] at (1.848,.7654) {$w'$};
\node at (2.2,0) {$1$};
\end{\tz}
\end{center}
\end{fig}

Note that for each of the domains $V$ in Figure \ref{fig1}, if a second open set $V'$ covers the interior boundary lines, then its function $d'$ in the first two figures would have to satisfy $d'(0,\frac12)=-\frac12$ in order to satisfy (\ref{C2}) and (\ref{symm}), while in the third figure, $d'$ must be 0 on a neighborhood of the two diagonal lines, and this will contradict (\ref{C1}) and (\ref{C2}).

\section{Proof of Theorem \ref{main}}
Our proof uses the following result of general topology. Throughout the paper, $\partial K$ is the boundary in the sense of general topology, sometimes called frontier.
\begin{prop} \label{conn} If $W$ is a connected bounded open set of the plane, and $K$ is the unbounded connected component of $\R^2-W$, then its boundary, $\partial K$, is connected.\end{prop}
\begin{proof}  \cite[Cor 1,p.~352]{Ba} states that a bounded connected open set in the plane has connected complement iff it has connected boundary, and calls this result ``well known.''
We apply this to $\R^2-K$, and note that $\partial(\R^2-K)=\partial K$.\end{proof}

A recent proof of this proposition appears in \cite{CKL}.
We will use the following corollary several times. It deals with a subspace $U$ of the unit square which is open in the subspace topology. By $\partial U$, we mean its boundary in $\R^2$.
\begin{cor}\label{cor} Let $U$ be a connected open subset of $I^2$ which intersects $\partial I^2$, and let $P$ and $Q$ be distinct points of the boundary in $\partial I^2$ of $U\cap\partial I^2$. Let $B$ and $B'$ be the two components of $\partial I^2-\{P,Q\}$. Suppose $B\cap U=\emptyset$. Then there is a connected subset of $\partial U-(B'\cap\partial U)$ which contains $P$ and $Q$.\end{cor}
\begin{proof} Apply Proposition \ref{conn} to $W=U-(U\cap\partial I^2)$, with $K$ being the unbounded component of $\R^2-W$. Then $\partial K$ is connected. Note that $\{P,Q\}\subset \partial K$ and $U\cap\partial I^2\subset\partial K$. The connected component of $\partial K-(U\cap\partial I^2)$ containing $P$ also contains $Q$ and is contained in $\partial U-(B'\cap\partial U)$.\end{proof}

\begin{proof}[Proof of Theorem \ref{main}] Suppose $I\times I$ is covered by two open sets $V$ and $V'$ equipped with locally constant functions $d$ and $d'$ satisfying (\ref{C1}), (\ref{C2}), and (\ref{symm}). We will show that this leads to a contradiction, which, along with the three symmetric motion planning rules described in Section \ref{sec2}, implies the theorem.

By compactness, only finitely many connected components of $V$ and $V'$ need be considered.
 At least one of these, say $V$, must contain $(0,0)$ and hence also the other three corner points. Schematically, there are three ways that the connected components of $V$ containing the corner points can occur, as illustrated in Figure \ref{three}, with the $d$-values in each indicated. We will call a portion such as occurs in the center of the second and third figures a ''bridge.'' A fourth possibility, a 90-degree rotation of the second figure, would have to have $d=0$ in all three parts by (\ref{symm}), and this would contradict (\ref{C1}) and (\ref{C2}). The small rectangular portions in the third figure are not a connected component of $V$ containing a corner point, but are necessary in order to satisfy (\ref{C1}) and (\ref{C2}).

 \smallskip

\begin{fig}\label{three}

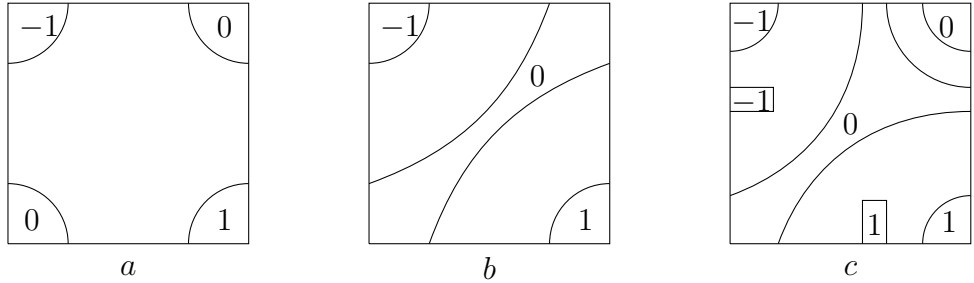
{\bf Three possibilities for corner-point neighborhoods.}

\begin{center}

\begin{\tz}[scale=3.2]
\draw (0,0) --(1,0) --(1,1) -- (0,1) -- (0,0);
\draw (.25,0) arc [radius=.25, start angle=0, end angle=90];
\draw (1,.25) arc [radius=.25, start angle=90, end angle=180];
\draw (0,.75) arc [radius=.25, start angle=270, end angle=360];
\draw (.75,1) arc [radius=.25, start angle=180, end angle=270];
\draw (3.8,1) arc [radius=.2, start angle=180, end angle=270];
\draw (1.5,.75) arc [radius=.25, start angle=270, end angle=360];
\draw (1.5,0) --(2.5,0) --(2.5,1) -- (1.5,1) -- (1.5,0);
\draw (2.5,.25) arc [radius=.25, start angle=90, end angle=180];
\draw (3,0) --(4,0) --(4,1) -- (3,1) -- (3,0);
\draw (1.75,0) to [out=70, in=200] (2.5,.75);
\draw (1.5,.25) to [out=20, in=250] (2.25,1);
\draw (4,.2) arc [radius=.2, start angle=90, end angle=180];
\draw (3,.8) arc [radius=.2, start angle=270, end angle=360];
\draw (3.65,1) arc [radius=.35, start angle=180, end angle=270];
\draw (3.2,0) to [out=70, in=180] (4,.55);
\draw (3,.2) to [out=20, in=270] (3.55,1);
\node at (.13,.9) {$-1$};
\node at (.1,.1) {$0$};
\node at (.9,.9) {$0$};
\node at (.9,.1) {$1$};
\node at (2.4,.1) {$1$};
\node at (2.2,.7) {$0$};
\node at (1.63,.9) {$-1$};
\node at (3.09,.92) {$-1$};
\node at (3.9,.9) {$0$};
\node at (3.6,.08) {$1$};
\node at (3.09,.59) {$-1$};
\node at (3.91,.09) {$1$};
\node at (3.5,.5) {$0$};
\draw (3.65,0) --(3.65,.18) --(3.55,.18) --(3.55,0);
\draw (3,.55) --(3.18,.55) --(3.18,.65) --(3,.65);
\node at (.5,-.1) {$a$};
\node at (2,-.1) {$b$};
\node at (3.5,-.1) {$c$};
\end{\tz}

\end{center}

\end{fig}

The $d$-values are as indicated because a connected open set containing a neighborhood of a point $(t,t)$ must have $d=0$ by (\ref{symm}), and then the other sets have $d=\pm1$ by (\ref{C1}) and (\ref{C2}). The boundaries of these regions need not be smooth curves as suggested by the diagram, but by Corollary \ref{cor} they are connected sets containing the two points on $\partial I^2$. These connected open sets can have holes, either in the interior of $I^2$ or, more significantly, containing an open set in $\partial I^2$. We will deal with the consequences of such holes later in the proof.

The other open set, $V'$, must contain connected open sets covering the boundaries of the components of $V$ just considered. There are several possibilities. We have indicated in orange a variety of such possibilities, again very schematically, in Figure \ref{four}. The $V'$ components may (cases $a_1$, $b_1$, $c_1$) or may not ($a_2$, $b_2$, $c_2$) contain the corner points. A neighborhood of the boundaries of a bridge may (cases $b_1$, $c_1$) or may not ($b_2$, $c_2$) contain the entire bridge.
%The two boundary sets in the upper right of the Figure \ref{three}c may be contained in separate $V'$ components or the same one.
A $V'$ neighborhood of the sets in  Figure \ref{three}a may (case $a_2$) or may not ($a_1$) contain a bridge. We list the $d$-values for the $V$-neighborhoods in Figure \ref{four} to emphasize which portions of the diagram are included in $V$.

\begin{minipage}{6.5in}
\begin{fig}\label{four}

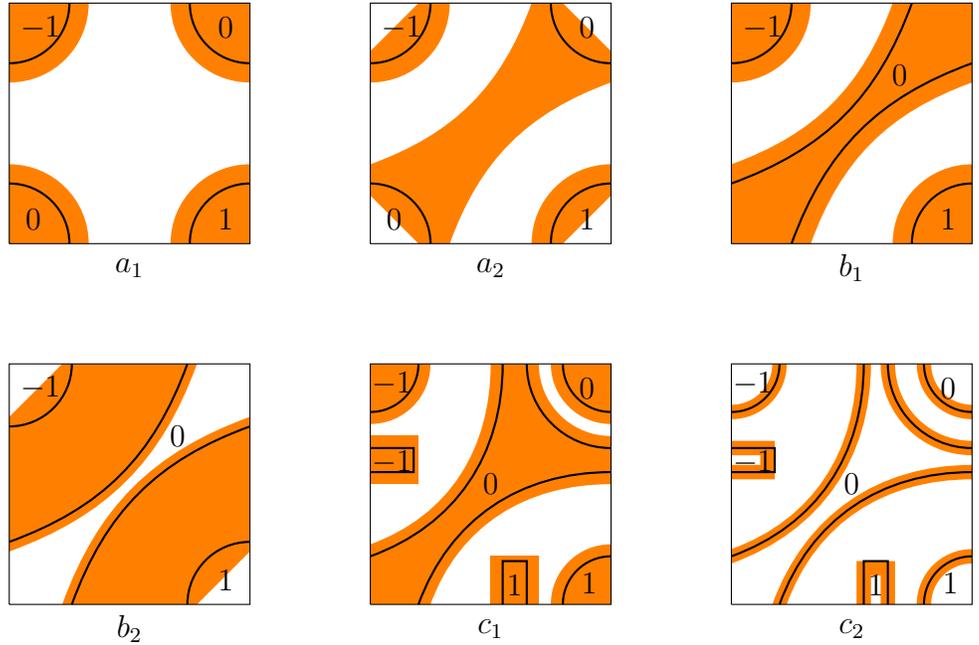
{\bf Six possibilities for first $V'$ neighborhoods.}

\begin{center}

\begin{\tz}[scale=3.2]
%\draw (0,1.5) --(1,1.5) --(1,2.5) -- (0,2.5) -- (0,1.5);
\path [fill=orange] (0,2.17) to [out=0,in=270] (.33,2.5) --(0,2.5) --(0,2.17);
\path [fill=orange] (.33,1.5) to [out=90,in=0] (0,1.83) --(0,1.5) --(.33,1.5);
\path [fill=orange] (.67,1.5) to [out=90,in=180] (1,1.83) --(1,1.5) --(.67,1.5);
\path [fill=orange] (1,2.17) to [out=180,in=270] (.67,2.5) --(1,2.5) --(1,2.17);
\draw (0,1.5) --(1,1.5) --(1,2.5) -- (0,2.5) -- (0,1.5);
\draw [thick] (.25,1.5) arc [radius=.25, start angle=0, end angle=90];
\draw [thick] (1,1.75) arc [radius=.25, start angle=90, end angle=180];
\draw [thick] (0,2.25) arc [radius=.25, start angle=270, end angle=360];
\draw [thick] (.75,2.5) arc [radius=.25, start angle=180, end angle=270];
\node at (.13,2.4) {$-1$};
\node at (.1,1.6) {$0$};
\node at (.9,2.4) {$0$};
\node at (.9,1.6) {$1$};
\node at (.5,1.4) {$a_1$};

\path [fill=orange] (1.5,2.17) to [out=0,in=270] (1.83,2.5) --(1.7,2.5) --(1.5,2.3) --(1.5,2.17);
\path [fill=orange] (2.17,1.5) to [out=90,in=180] (2.5,1.83) --(2.5,1.7) --(2.3,1.5) --(2.17,1.5);
\path [fill=orange] (1.83,1.5) to [out=70,in=200] (2.5,2.17) --(2.5,2.3) --(2.3,2.5) --(2.17,2.5) to [out=250,in=20] (1.5,1.83) --(1.5,1.7) --(1.7,1.5) --(1.83,1.5);
\draw (1.5,1.5) --(2.5,1.5) --(2.5,2.5) -- (1.5,2.5) -- (1.5,1.5);
\draw [thick] (1.75,1.5) arc [radius=.25, start angle=0, end angle=90];
\draw [thick] (2.5,1.75) arc [radius=.25, start angle=90, end angle=180];
\draw [thick] (1.5,2.25) arc [radius=.25, start angle=270, end angle=360];
\draw [thick] (2.25,2.5) arc [radius=.25, start angle=180, end angle=270];
\node at (1.63,2.4) {$-1$};
\node at (1.6,1.6) {$0$};
\node at (2.4,2.4) {$0$};
\node at (2.4,1.6) {$1$};
\node at (2,1.4) {$a_2$};

\path [fill=orange] (3,2.17) to [out=0,in=270] (3.33,2.5) --(3,2.5) --(3,2.17);
\path [fill=orange] (3.67,1.5) to [out=90,in=180] (4,1.83) --(4,1.5) --(3.67,1.5);
\path [fill=orange] (3.33,1.5) to [out=70,in=200] (4,2.17) --(4,2.5) --(3.67,2.5) to [out=250,in=20] (3,1.83) --(3,1.5) --(3.33,1.5);
\draw [thick] (3.25,1.5) to [out=70, in=200] (4,2.25);
\draw [thick] (3,1.75) to [out=20, in=250] (3.75,2.5);
\draw (3,1.5) --(4,1.5) --(4,2.5) -- (3,2.5) -- (3,1.5);
\draw [thick] (4,1.75) arc [radius=.25, start angle=90, end angle=180];
\draw [thick] (3,2.25) arc [radius=.25, start angle=270, end angle=360];
\node at (3.9,1.6) {$1$};
\node at (3.7,2.2) {$0$};
\node at (3.13,2.4) {$-1$};
\node at (3.5,1.4) {$b_1$};

\path [fill=orange] (.22,0) --(.78,0) --(1,.22) --(1,.78) to [out=200,in=70] (.22,0);
\path [fill=orange] (0,.22) to [out=20,in=250] (.78,1) --(.22,1) --(0,.78) --(0,.22);
\draw (0,0) --(1,0) --(1,1) --(0,1) --(0,0);
\draw [thick] (1,.26) arc [radius=.26, start angle=90, end angle=180];
\draw [thick] (0,.74) arc [radius=.26, start angle=270, end angle=360];
\draw [thick] (.26,0) to [out=70, in=200] (1,.74);
\draw [thick] (0,.26) to [out=20, in=250] (.74,1);
\node at (.9,.1) {$1$};
\node at (.7,.7) {$0$};
\node at (.13,.9) {$-1$};
\node at (.5,-.1) {$b_2$};

\path [fill=orange] (1.5,.75) to [out=0,in=270] (1.75,1) --(1.5,1) --(1.5,.75);
\path [fill=orange] (2.25,0) to [out=90,in=180] (2.5,.25) --(2.5,0) --(2.25,0);
\path [fill=orange] (2.5,.75) to [out=180,in=270] (2.25,1) --(2.5,1) --(2.5,.75);
\path [fill=orange] (1.75,0) to [out=70, in=180] (2.5,.5) --(2.5,.7) to [out=180,in=270] (2.2,1) --(2,1) to [out=270,in=20] (1.5,.25) --(1.5,0) --(1.75,0);
\path [fill=orange] (2,0) --(2,.2) --(2.2,.2) --(2.2,0) --(2,0);
\path [fill=orange] (1.5,.5) --(1.7,.5) --(1.7,.7) --(1.5,.7) --(1.5,.5);
\draw (1.5,0) --(2.5,0) --(2.5,1) -- (1.5,1) -- (1.5,0);
\draw [thick] (2.5,.2) arc [radius=.2, start angle=90, end angle=180];
\draw [thick] (1.5,.8) arc [radius=.2, start angle=270, end angle=360];
\draw [thick] (2.15,1) arc [radius=.35, start angle=180, end angle=270];
\draw [thick] (1.7,0) to [out=70, in=180] (2.5,.55);
\draw [thick] (1.5,.2) to [out=20, in=270] (2.05,1);
\draw [thick] (2.3,1) arc [radius=.2, start angle=180, end angle=270];
\draw [thick] (2.05,0) --(2.05,.18) --(2.15,.18) --(2.15,0);
\draw [thick] (1.5,.55) --(1.68,.55) --(1.68,.65) --(1.5,.65);
\node at (1.59,.92) {$-1$};
\node at (2.4,.9) {$0$};
\node at (2.1,.08) {$1$};
\node at (1.59,.59) {$-1$};
\node at (2.41,.09) {$1$};
\node at (2,.5) {$0$};
\node at (2,-.1) {$c_1$};

\path [fill=orange] (4,.62) to [out=180,in=270] (3.62,1) --(3.68,1) to [out=270,in=180] (4,.68) --(4,.62);
\path [fill=orange] (3,.17) to [out=20, in=270] (3.58,1) --(3.52,1) to [out=270,in=20] (3,.23) --(3,.17);
\path [fill=orange] (4,.77) to [out=180,in=270] (3.77,1) --(3.83,1) to [out=270,in=180] (4,.83) --(4,.77);
\path [fill=orange] (3,.52) --(3.18,.52) --(3.18,.68) --(3,.68) --(3,.62) --(3.12,.62) --(3.12,.58) --(3,.58) --(3,.52);
\path [fill=orange] (3,.77) to [out=0,in=270] (3.23,1) --(3.17,1) to [out=270,in=0] (3,.83) --(3,.77);
\path [fill=orange] (3.77,0) to [out=90,in=180] (4,.23) --(4,.17) to [out=180,in=90] (3.83,0) --(3.77,0);
\path [fill=orange] (3.52,0) --(3.52,.18) --(3.68,.18) --(3.68,0) --(3.62,0) --(3.62,.12) --(3.58,.12) --(3.58,0) --(3.52,0);
\path [fill=orange] (3.23,0) to [out=70,in=180] (4,.52) --(4,.58) to [out=180,in=70] (3.17,0) --(3.23,0);
\draw (3,0) --(4,0) --(4,1) -- (3,1) -- (3,0);
\draw [thick] (4,.2) arc [radius=.2, start angle=90, end angle=180];
\draw [thick] (3,.8) arc [radius=.2, start angle=270, end angle=360];
\draw [thick] (3.65,1) arc [radius=.35, start angle=180, end angle=270];
\draw [thick] (3.2,0) to [out=70, in=180] (4,.55);
\draw [thick] (3,.2) to [out=20, in=270] (3.55,1);
\draw [thick] (3.8,1) arc [radius=.2, start angle=180, end angle=270];
\draw [thick] (3.55,0) --(3.55,.18) --(3.65,.18) --(3.65,0);
\draw [thick] (3,.55) --(3.18,.55) --(3.18,.65) --(3,.65);
\node at (3.09,.92) {$-1$};
\node at (3.9,.9) {$0$};
\node at (3.6,.08) {$1$};
\node at (3.09,.59) {$-1$};
\node at (3.91,.09) {$1$};
\node at (3.5,.5) {$0$};
\node at (3.5,-.1) {$c_2$};
\end{\tz}

\end{center}

\end{fig}
\end{minipage}

Note that only Figure \ref{four}$b_2$ has the property that the displayed sets $V$ and $V'$ cover $I^2$.  However, in this case, $V'$ does not admit a $d'$ function. If $C_L$ and $C_R$ are its two components, then $d'(C_R)=-d'(C_L)$ by (\ref{symm}) but $d'(C_R)=d'(C_L)+1$ by (\ref{C2}), and so these cannot have integer values. In all the other cases of Figure \ref{four}, $V'$ admits a $d'$-function which equals the $d$-value of the $V$-set which it intersects.

Next, $V$ must contain open connected sets containing all of the $V'$-boundaries just obtained which were not contained in the initial $V$-sets. Note that the boundary of a $V$-component can never intersect the boundary of a $V'$-component because (since $V$ and $V'$ are open) the boundary points are not in the open sets, and so an intersection point of the two boundaries would not be in $V\cup V'$, which is supposed to cover the square.

In order to have $V\cup V'$ cover the square, we must continue alternately adding new components of $V$ and $V'$, each time covering new boundary parts of the other set just added. At some stage, the situation illustrated in \ref{three}$a$ and \ref{four}$a_1$ must yield to a bridge, in order that the diagonal is covered by $V\cup V'$. At some stage, the bands coming down to the right from the bridge, and up from the lower right corner must combine. For example, that could happen with the next $V$-component in Figure \ref{four} parts $a_2$, $b_1$, $c_1$, or $c_2$. We claim that this will necessarily cause a contradiction on the $d$-function similar to that observed in Figure \ref{four}$b_2$.

We consider now the $V$-boundaries which extend from the lower edge to the right edge, as well as their symmetric counterparts. These, as well as the analogous $V'$-boundaries, are obtained iteratively using Corollary \ref{cor} with $P$ equal to the sup of the set of $x$ such that $(x,0)$ is in the closure of the $V$ or $V'$ component being considered, and $Q$ the inf of $y$'s such that $(1,y)$ is in this closure. The $V$-boundaries are illustrated schematically in Figure \ref{five}.

\begin{fig}\label{five}

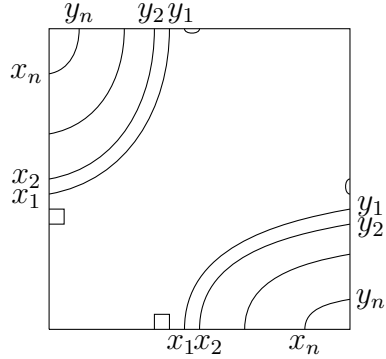
{\bf Boundaries of $V$-bands.}

\begin{center}

\begin{\tz}[scale=2]
\draw (0,0) --(2,0) --(2,2) --(0,2) --(0,0);
\draw (.9,0) to [out=90,in=190] (2,.8);
\draw (1,0) to [out=90,in=190] (2,.7);
\draw (1.3,0) to [out=90,in=190] (2,.5);
\draw(1.7,0) to [out=90,in=190] (2,.2);
\node at (.88,-.1) {$x_1$}; \node at (-.15,.87) {$x_1$};
\node at (1.06,-.1) {$x_2$}; \node at (-.15,1) {$x_2$};
\node at (1.7,-.1) {$x_n$}; \node at (-.15,1.7) {$x_n$};
\draw (0,.9) to [out=10,in=270] (.8,2);
\draw (0,1) to [out=10,in=270] (.7,2);
\draw (0,1.3) to [out=10,in=270] (.5,2);
\draw (0,1.7) to [out=10,in=270] (.2,2);
\node at (2.14,.82) {$y_1$}; \node at (.88,2.1) {$y_1$};
\node at (2.14,.68) {$y_2$}; \node at (.68,2.1) {$y_2$};
\node at (2.14,.2) {$y_n$}; \node at (.2,2.1) {$y_n$};
\draw (.9,2) to [out=270,in=270] (1,2);
\draw (2,.9) to [out=180,in=180] (2,1);
\draw (0,.7) --(.1,.7) --(.1,.8) --(0,.8);
\draw (.7,0) --(.7,.1) --(.8,.1) --(.8,0);
\end{\tz}
\end{center}
\end{fig}

The open bands between these boundary sets could have interior holes, which are not an issue at all, as a hole in a $V$-band can be covered by a $V'$ open set with no problem regarding the $d$-function, and vice versa, and they can have holes or modifications at the boundary, which we will consider later. Temporarily ignoring this possibility, each of the open bands between consecutive boundary sets must be either a component of $V$, or else covered by a component of $V'$. The requirement that $(0,t)\in V$ iff $(1,t)$ forces additional components as illustrated in the tiny portions of Figure \ref{five}.

The region between the two $x_1y_1$ boundary sets must have $d=0$ by (\ref{symm}). If $y_1< x_1$, as is the case in Figure \ref{five}, then for $t$ between $y_1$ and $x_1$, we obtain a contradiction to (\ref{C2}).
If $y_1=x_1$, then the $V'$ sets containing the two $x_1y_1$ boundary sets imply $d'(0,x_1)=-\frac12$ by (\ref{C2}) and (\ref{symm}), contradicting integrality of $d'$.
At the other extreme, if $y_n>x_n$, then if $x_n<t<y_n$, then $(0,t)$ and $(1,t)$ lie in symmetric regions, so $d(0,t)=-d(1,t)$, which is inconsistent with (\ref{C2}) and integrality of $d$. See Figure \ref{seven} for an illustration of a variation on this.

Otherwise, let $k$ be minimal such that $y_k\le x_k$. The preceding paragraph shows that such a $k$ must exist with $1<k<n$. If $y_k<x_k$, then there exists $t$ satisfying $x_{k-1}<t<x_k$ and $y_k<t<y_{k-1}$. Then $(0,t)$ and $(1,t)$ lie in symmetric bands, and so (\ref{symm}) and (\ref{C2}) imply the usual contradiction to integrality of $d$. If $y_k=x_k$, the contradiction is obtained on $d'(0,x_k)$ and $d'(1,x_k)$ using the $V'$ sets which contain the $x_ky_k$ boundary set and its symmetric counterpart.

Our contradictions have all been due to points of $\partial I^2$ which lie in symmetric bands of components of $V$ or $V'$. It is conceivable that these bands might have a hole where they meet $\partial I^2$.  Suppose $y_k<x_{k-1}<y_{k-1}<x_k$, so that we expect to obtain a contradiction in this band. It could happen that the points $(0,t)$ for $x_{k-1}<t<y_{k-1}$ are cut off from the main band as indicated schematically in Figure \ref{six}, in which we write $k'$ instead of $k-1$ for typographical reasons.

\begin{fig}\label{six}

{\bf Hole at contradiction point.}

\begin{center}

\begin{\tz}[scale=2]
\draw (0,0) --(2,0) --(2,2) --(0,2) --(0,0);
\path [fill=orange] (1.3,0) to [out=140,in=130] (2,.7) --(2,.86) --(1.9,1) --(2,1.14) --(2,1.3) to [out=220,in=90] (.7,0) --(.86,0) to [out=45,in=135] (1.14,0) --(1.3,0);
\draw [thick] (2,.8) --(1.8,1) --(2,1.2);
\draw [thick] (.8,0) to [out=75,in=180] (1,.2) to [out=0,in=105] (1.2,0);
\draw [thick] (2,1.2) to [out=220,in=90] (.8,0);
\draw [thick] (1.4,0) to [out=90,in=180] (2,.6);
\node at (1.94,1) {$S$};
\node at (.8,-.1) {$x_{k'}$};
\node at (1.2,-.1) {$y_{k'}$};
\node at (1.45,-.1) {$x_k$};
\node at (2.12,.6) {$y_k$};
\node at (2.15,.8) {$x_{k'}$};
\node at (2.15,1.2) {$y_{k'}$};
\path [fill=orange] (0,1.3) to [out=310,in=320] (.7,2) --(.86,2) --(1,1.9) --(1.14,2) --(1.3,2) to [out=230,in=0] (0,.7) --(0,.86) to [out=45,in=315] (0,1.14) --(0,1.3);
\draw [thick] (.8,2) --(1,1.8) --(1.2,2);
\draw [thick] (0,.8) to [out=15,in=270] (.2,1) to [out=90,in=345] (0,1.2);
\draw [thick] (1.2,2) to [out=230,in=0] (0,.8);
\draw [thick] (0,1.4) to [out=0,in=270] (.6,2);
\node at (.1,1) {$R$};
\node at (-.13,.8) {$x_{k'}$};
\node at (-.13,1.2) {$y_{k'}$};
\node at (-.13,1.4) {$x_k$};
\node at (.56,2.1) {$y_k$};
\node at (.81,2.1) {$x_{k'}$};
\node at (1.23,2.1) {$y_{k'}$};
\end{\tz}
\end{center}
\end{fig}

The indicated region $R$ and its symmetric counterpart must be part of $V$ or part of $V'$. Since $(t,0)\in V$ (resp.~$V'$) iff $(t,1)\in V$ (resp.~$V'$), there will be an opposing part of $V$ or $V'$, as suggested by the set $S$ in Figure \ref{six}, which will also have a symmetric counterpart. The boundary of $R$ must intersect the $x_{k-1}y_{k-1}$ boundary, for otherwise there would be points of $\partial I^2$ in the band  giving the previous contradiction. There must be a $V'$ component containing the union of $\partial R$ and the $x_{k-1}y_{k-1}$ boundary, as indicated by the orange set in Figure \ref{six}. The values $d'(0,x_{k-1})$ and $d'(1,x_{k-1})$ give the usual contradiction ((\ref{symm}), (\ref{C2}), and integrality).

A similar situation could occur regarding the contradiction that was obtained earlier in the case that $y_n>x_n$. As illustrated in Figure \ref{seven}, the $t$-values between $x_n$ and $y_n$ on $\partial I^2$ could be cut off from the main symmetric regions. We obtain the same sort of contradiction as was obtained in Figure \ref{six}, using here the $V'$-component containing the union of the $x_ny_n$-boundary and $\partial R$, and its symmetric counterpart.
\end{proof}

\begin{fig}\label{seven}

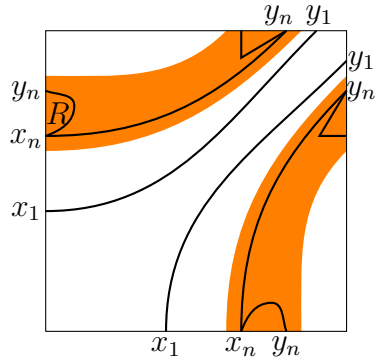
{\bf Hole at contradiction point near end.}

\begin{center}

\begin{\tz}[scale=2]
\path [fill=orange] (2,1.7) to [out=225,in=90] (1.2,0) --(1.7,0) to [out=90,in=225] (2,1.2) --(2,1.7);
\path [fill=orange] (1.7,2) to [out=225,in=0] (0,1.2) --(0,1.7) to [out=0,in=225] (1.2,2) --(1.7,2);
\draw (0,0) --(2,0) --(2,2) --(0,2) --(0,0);
\draw [thick] (.8,0) to [out=90,in=225] (2,1.8);
\draw [thick] (0,.8) to [out=0,in=225] (1.8,2);
\draw [thick] (1.3,0) to [out=90,in=225] (2,1.6);
\draw [thick] (0,1.3) to [out=0,in=225] (1.6,2);
\draw [thick] (1.3,0) to [out=70,in=180] (1.5,.19) to [out=0,in=110] (1.6,0);
\draw [thick] (0,1.3) to [out=20,in=270] (.19,1.5) to [out=90,in=340] (0,1.6);
\draw [thick] (2,1.6) --(1.82,1.3) --(2,1.3);
\draw [thick] (1.6,2) --(1.3,1.82) --(1.3,2);
\node at (.8,-.1) {$x_1$};
\node at (1.3,-.1) {$x_n$};
\node at (1.6,-.1) {$y_n$};
\node at (-.13,.8) {$x_1$};
\node at (-.13,1.28) {$x_n$};
\node at (-.13,1.6) {$y_n$};
\node at (.08,1.45) {$R$};
\node at (2.1,1.8) {$y_1$};
\node at (2.1,1.6) {$y_n$};
\node at (1.55,2.1) {$y_n$};
\node at (1.82,2.1) {$y_1$};
\end{\tz}
\end{center}
\end{fig}
 \def\line{\rule{.6in}{.6pt}}

\end{document}